\documentclass[11pt]{article}
\usepackage[T1]{fontenc}
\usepackage{lmodern}
\usepackage{amsmath,amssymb,amsthm,mathtools}
\usepackage[a4paper,margin=29mm]{geometry}
\usepackage{microtype}
\usepackage{enumitem}
\usepackage{booktabs,array}
\usepackage[colorlinks=true,linkcolor=blue!50!black,citecolor=blue!50!black,urlcolor=blue!50!black]{hyperref}
\usepackage{xcolor}
\usepackage{fancyhdr}
\newtheorem{theorem}{Theorem}[section]
\newtheorem{proposition}[theorem]{Proposition}
\newtheorem{lemma}[theorem]{Lemma}
\newtheorem{corollary}[theorem]{Corollary}
\theoremstyle{definition}\newtheorem{definition}[theorem]{Definition}
\newtheorem{example}[theorem]{Example}
\theoremstyle{remark}\newtheorem{remark}[theorem]{Remark}
\DeclareMathOperator{\Tor}{Tor}
\DeclareMathOperator{\fd}{fd}
\DeclareMathOperator{\Fwd}{Fwd}
\DeclareMathOperator{\wgd}{w.gl.dim}
\DeclareMathOperator{\Ann}{Ann}
\DeclareMathOperator{\Nil}{Nil}

\DeclareMathOperator{\Max}{Max}
\DeclareMathOperator{\Frac}{Frac}
\DeclareMathOperator{\im}{im}
\newcommand{\m}{\mathfrak m}
\newcommand{\p}{\mathfrak p}
\newcommand{\q}{\mathfrak q}
\newcommand{\kk}{\kappa}
\newcommand{\ot}{\otimes}
\newcommand{\colim}{\varinjlim}

\numberwithin{equation}{section}
\title{The finitistic weak dimensions of Gaussian rings}
\author{Xiaolei Zhang \thanks{School of Mathematics and Statistics, Tianshui Normal University
		Tianshui 741000, China;	E-mail: zxlrghj@163.com }}
\date{}
\begin{document}
\maketitle
\begin{abstract}
We prove that every commutative Gaussian ring has finitistic weak dimension at most two. The proof works with a fixed module of finite flat dimension and does not assume that the weak global dimension of the ring is finite. For a local Gaussian ring $A$, with nilradical $N$, the kernel $I$ of $A\to A_N$ satisfies $NI=I^2=0$. A hypothetical module $M$ of flat dimension three gives a nonzero torsion module $H=\Tor_2^A(A/N,M)$ over the valuation domain $A/N$, with $H\cong\Tor_1^A(I,M)$ and $I\otimes_{A/N}H=0$. Two carefully controlled localizations, a propagation argument in the chain ring $A/D$, where $D=\{x:x^2=0\}$, and an annihilator factorization in a filtered system of cyclic modules yield a contradiction. 
\end{abstract}

\begin{abstract}
We prove that every commutative Gaussian ring has finitistic weak dimension at most two. The proof works with a fixed module of finite flat dimension and does not assume that the weak global dimension of the ring is finite. For a local Gaussian ring $A$, with nilradical $N$, the kernel $I$ of $A\to A_N$ satisfies $NI=I^2=0$. A hypothetical module $M$ of flat dimension three gives a nonzero torsion module $H=\Tor_2^A(A/N,M)$ over the valuation domain $A/N$, with $H\cong\Tor_1^A(I,M)$ and $I\otimes_{A/N}H=0$. Two carefully controlled localizations, a propagation argument in the chain ring $A/D$, where $D=\{x:x^2=0\}$, and an annihilator factorization in a filtered system of cyclic modules yield a contradiction. An explicit valuation-ring quotient realizes the bound two. We also identify the precise inputs from the structure theory of Gaussian rings and explain why passage to a square-zero quotient cannot replace the Tor arguments.
\end{abstract}
\maketitle

\setcounter{tocdepth}{1}

\section{Introduction}

For a commutative ring $R$, the weak global dimension measures the flat dimensions of all $R$-modules, whereas the finitistic weak dimension measures only the finite ones:
\[
 \Fwd(R)=\sup\{\fd_R M: M\text{ is an }R\text{-module and }\fd_R M<\infty\}.
\]
No finite generation or presentation condition is imposed on $M$. The distinction is essential for noncoherent rings. General background on finitistic dimensions, coherent rings, and homological algebra can be found in \cite{Bass,GlazBook,Weibel}.

A ring is Gaussian if the content of a product of polynomials is the product of their contents. Gaussian domains are precisely Pr\"ufer domains. In the presence of zero divisors, Gaussian rings need not be arithmetical or coherent. Their local structure nevertheless retains a chain-ring quotient and a valuation-domain quotient, which provide two different kinds of control over ideals.

Bazzoni and Glaz developed the relevant structure theory in \cite{BG}. Donadze and Thomas proved the weak-global-dimension trichotomy in \cite{DT}: a Gaussian ring has weak global dimension $0$, $1$, or $\infty$. This trichotomy does not by itself bound the finite flat dimensions of individual modules when the weak global dimension is infinite. Couchot proved the finitistic bound two for commutative arithmetical rings in \cite{Couchot12}, and subsequently for fqf-rings in \cite{Couchot15}. The following result gives the corresponding bound for Gaussian rings.

\begin{theorem}\label{thm:main}
Let $R$ be a commutative Gaussian ring. Every $R$-module of finite flat dimension has flat dimension at most two. Equivalently,
\[
 \Fwd(R)\leq 2.
\]
The bound can be attained by a local Gaussian total ring of quotients.
\end{theorem}

The main issue is local. For a local Gaussian ring $A$, put $N=\Nil(A)$ and $D=\{x\in A:x^2=0\}$. Although $A/D$ is a chain ring, it need not be reduced. In contrast, $A/N$ is a valuation domain and has weak global dimension at most one. These two quotients have different roles in the proof. The chain ring $A/D$ propagates nonvanishing between cyclic modules; the valuation domain $A/N$ controls a change-of-rings spectral sequence and the annihilators of torsion elements.

The ideal used for the square-zero part of the argument is not the whole nilradical. It is
\[
 I=\ker(A\longrightarrow A_N).
\]
We show that $NI=0$, even if $N^2\ne0$. For a hypothetical module of flat dimension three, the obstruction module $H=\Tor_2^A(A/N,M)$ is identified with $\Tor_1^A(I,M)$. A first localization makes a nonzero element of $H$ power-torsion under every element of the new maximal ideal. A second localization does the same for a suitable element of $I$. Preservation of the required Tor nonvanishing at the second localization is proved explicitly.

Section~\ref{sec:prelim} records the homological facts and the precise external Gaussian-ring inputs. Section~\ref{sec:kernel} studies $I$ and $H$. Sections~\ref{sec:normal}--\ref{sec:limit} prove the local contradiction. Section~\ref{sec:global} gives the global conclusion, and Section~\ref{sec:example} supplies a sharp example. The final section records several points where a quotient or localization argument would otherwise be incomplete.

\section{Preliminaries}\label{sec:prelim}

All rings are commutative with identity, and all modules are unital. For a module $X$, write $\Ann_R(X)$ for its annihilator; for an element $x$, also write $(0:x)=\Ann_R(x)$. A chain ring is a ring whose ideals are linearly ordered by inclusion. A valuation domain is a chain ring that is an integral domain. We use standard valuation-domain module theory as in \cite{FS}.

The local notation used throughout the proof is collected below. The letter $k$ denotes the residue field, whereas $K$ denotes the fraction field of the reduced quotient; they should not be identified.
\begin{center}
\begin{tabular}{@{}ll@{}}
\toprule
Symbol & Meaning \\
\midrule
$(A,\m)$ & A local Gaussian ring and its maximal ideal \\
$N$ & The nilradical $\Nil(A)$ \\
$D$ & The ideal $\{x\in A:x^2=0\}$ \\
$V,\p,k$ & $A/N$, $\m/N$, and $A/\m$ \\
$S,Q,K$ & $A\setminus N$, $A_N$, and $\Frac(V)$ \\
$I$ & The kernel of $A\to A_N$ \\
$T_i(X),H$ & $\Tor_i^A(X,M)$ and $\Tor_2^A(V,M)$ \\
\bottomrule
\end{tabular}
\end{center}

For $f\in R[t]$, its content $c(f)$ is the ideal generated by its coefficients. A ring $R$ is Gaussian if $c(fg)=c(f)c(g)$ for every $f,g\in R[t]$. Quotients and localizations of Gaussian rings are Gaussian, and the Gaussian property can be tested at maximal ideals.

\subsection{Flat dimension, syzygies, and finitistic conventions}

\begin{definition}
An $R$-module $F$ is \emph{flat} if tensoring with $F$ preserves short exact sequences. The flat dimension $\fd_R M$ is the smallest integer $n\geq0$ for which there is an exact sequence
\[
 0\longrightarrow F_n\longrightarrow F_{n-1}\longrightarrow\cdots
 \longrightarrow F_0\longrightarrow M\longrightarrow0
\]
with every $F_i$ flat. It is infinity if no such integer exists. We assign flat dimension zero to the zero module. The weak global dimension is
\[
 \wgd(R)=\sup\{\fd_R M:M\text{ an }R\text{-module}\}.
\]
The invariant $\Fwd(R)$, which is called to be \emph{ finitistic weak dimension} of $R$, restricts this supremum to the modules of finite flat dimension. 
\end{definition}

We shall use the ideal criterion for flatness: an $R$-module $F$ is flat precisely when $J\ot_R F\to F$ is injective for every finitely generated ideal $J$. Equivalently, $\Tor_1^R(R/J,F)=0$ for all such $J$. This criterion and the standard construction of free resolutions are valid without coherence; see \cite{GlazBook,Weibel}. The following useful consequence explains both the reduction to dimension three and the cyclic tests used later.

\begin{proposition}\label{prop:syzygy-details}
For an $R$-module $M$ and $n\geq0$, the following are equivalent:
\begin{enumerate}
\item $\fd_R M\leq n$;
\item $\Tor_{n+1}^R(X,M)=0$ for every $R$-module $X$;
\item $\Tor_{n+1}^R(R/J,M)=0$ for every finitely generated ideal $J$;
\item an $n$th syzygy in a free resolution of $M$ is flat.
\end{enumerate}
If $\fd_R M=n>0$, then an $r$th syzygy, for $0\leq r\leq n$, has flat dimension $n-r$.
\end{proposition}
\begin{proof}
Let $K_0=M$, and choose exact sequences
\[
 0\longrightarrow K_{i+1}\longrightarrow F_i\longrightarrow K_i
 \longrightarrow0
\]
with $F_i$ free. The Tor long exact sequence gives
\[
 \Tor_j^R(X,K_{i+1})\cong\Tor_{j+1}^R(X,K_i)
 \quad(j\geq1).
\]
Iteration identifies $\Tor_1^R(R/J,K_n)$ with $\Tor_{n+1}^R(R/J,M)$. Thus (iii) implies (iv) by the ideal criterion. Condition (iv) gives a flat resolution of length $n$, proving (i). Computing Tor from such a resolution proves (ii), which plainly implies (iii).

The same dimension-shifting identities show $\fd_R K_r\leq n-r$. If the inequality were strict for $r<n$, splicing a shorter flat resolution of $K_r$ to the displayed free sequences would give $\fd_R M<n$. For $r=n$, flatness gives the asserted value zero with our convention.
\end{proof}

\begin{remark}
The two uses of highest Tor in the proof are logically distinct. Flat dimension exactly three provides at least one nonzero third Tor. Flat dimension at most three makes fourth Tor vanish and therefore makes third Tor preserve monomorphisms. The first property detects an obstruction; the second transfers it along embeddings.
\end{remark}

\subsection{Related ring-theoretic concepts}

A ring is \emph{coherent} if every finitely generated ideal is finitely presented. For a coherent ring, annihilators of individual elements are finitely generated. A domain is \emph{Pr\"ufer} if each nonzero finitely generated ideal is invertible. A commutative ring is \emph{arithmetical} if its localizations at maximal ideals are chain rings. These definitions distinguish the hypotheses appearing in the results discussed in the introduction: the Gaussian hypothesis does not include coherence, and our sharp example is not coherent.

An element is \emph{regular} if multiplication by it on the ring is injective. A \emph{total ring of quotients} is a ring in which every regular element is a unit. This terminology does not assert that the ring is a field or that its finitistic weak dimension is zero. The example in Section~\ref{sec:example} makes both distinctions explicit.

An ideal $J$ is \emph{nil} if each of its elements is nilpotent. It is \emph{nilpotent} if $J^r=0$ for one integer $r$ independent of the element. A nil ideal need not be nilpotent. In particular, the nilradical $N$ of a local Gaussian ring need not have bounded nilpotency degree. We never impose such a bound on $N$.

The ideal $D$ is more restrictive: its elements satisfy the particular equation $x^2=0$. In a general commutative ring, the square-zero elements need not form an ideal with square zero. The assertions that they do so here, and that $A/D$ is a chain ring, are substantive parts of Gaussian structure theory. Similarly, $I=\ker(A\to A_N)$ is not introduced as an arbitrary square-zero ideal; its annihilator and divisibility properties will be derived from the Gaussian hypothesis.

For context, the fqf condition used in \cite{Couchot15} says that a finitely generated ideal $J$ is flat as a module over $R/\Ann_R(J)$. It supplies an earlier finitistic bound for a related class of rings. We do not assume this condition or use it to replace Gaussian structure in the argument below.

\subsection{Valuation-domain facts used in the localizations}

\begin{lemma}\label{lem:valuation-details}
Let $V$ be a valuation domain, with maximal ideal $\p$ and fraction field $K$.
\begin{enumerate}
\item Every finitely generated ideal of $V$ is principal. Every ideal is the directed union of its principal subideals.
\item Every torsion-free $V$-module is flat; consequently $\wgd(V)\leq1$.
\item The radical of a proper ideal of $V$ is prime.
\item If $\q$ is prime, $c\in\q$, and $s\notin\q$, then $c\in sV$.
\item If $\p$ is not principal and $0\ne b\in\p$, then the principal ideals $uV$ with $bV\subseteq uV\subsetneq\p$ have union $\p$, and none is a largest member.
\end{enumerate}
\end{lemma}
\begin{proof}
For (i), a finite linearly ordered set of principal ideals has a largest member, which equals the ideal generated by all the elements. The union assertion follows by applying this observation to any finite subset of an ideal.

For (ii), let $L$ be torsion-free. Every nonzero finitely generated ideal is $aV\cong V$. Under this isomorphism the map $aV\ot_V L\to L$ is multiplication by $a$, which is injective. The ideal criterion gives flatness. For an arbitrary module, the kernel of a surjection from a free module is torsion-free, and hence flat. This gives a flat resolution of length one.

For (iii), suppose $ab\in\sqrt J$. Choose $n$ with $(ab)^n\in J$. By comparability, either $a=br$ or $b=ar$ for some $r\in V$. In the first case $a^{2n}=(ab)^nr^n\in J$, and in the second $b^{2n}\in J$. Thus one factor belongs to $\sqrt J$.

For (iv), comparability gives either $cV\subseteq sV$ or $sV\subseteq cV$. The latter would imply $s\in\q$. Finally, given $c\in\p$, the larger of $bV$ and $cV$ is a member of the family in (v). If any member contained all the others, their union would be that principal ideal, contradicting nonprincipality of $\p$.
\end{proof}

\begin{lemma}\label{lem:torsion-localization}
Let $V$ be a valuation domain, $H$ be a $V$-module, and let $0\ne h\in H$ have nonzero annihilator. Put $\q=\sqrt{\Ann_V(h)}$. Then $h/1\ne0$ in $H_{\q}$ and every $c\in\q V_{\q}$ has a power annihilating $h/1$. A module element $z$ need not survive at $\q$ merely because it is nonzero. It does survive if $cz\ne0$ for some $c\in\q$ and $V$ is a valuation domain.
\end{lemma}
\begin{proof}
The annihilator is proper, so its radical is a proper prime by Lemma~\ref{lem:valuation-details}. If $h/1=0$, some $s\notin\q$ would annihilate $h$, contradicting $\Ann_V(h)\subseteq\q$. For $c=a/s\in\q V_{\q}$, a power of $a\in\q$ belongs to $\Ann_V(h)$, which proves the power-annihilation assertion. Finally, if $s\notin\q$ annihilated $z$, then $c\in sV$ by the preceding lemma, and hence $cz=0$.
\end{proof}

\subsection{Elementary homological facts}

\begin{lemma}\label{lem:top}
Let $M$ be an $A$-module with $\fd_A M\leq n$. Then $\Tor_n^A(-,M)$ preserves monomorphisms. If $B$ embeds into a flat $A$-module and $n\geq1$, then $\Tor_n^A(B,M)=0$.
\end{lemma}
\begin{proof}
For $0\to U\to W\to C\to0$, the relevant exact sequence begins with
\[
 \Tor_{n+1}^A(C,M)\longrightarrow\Tor_n^A(U,M)
 \longrightarrow\Tor_n^A(W,M).
\]
The first term is zero. The second assertion follows by taking $W$ flat.
\end{proof}

\begin{lemma}\label{lem:localtor}
Let $S$ be multiplicatively closed and $Q=S^{-1}A$. If $X$ is a $Q$-module, then
\[
 \Tor_i^A(X,M)\cong\Tor_i^Q(X,S^{-1}M)\quad(i\geq0).
\]
Moreover, localization commutes with Tor in both variables.
\end{lemma}
\begin{proof}
Localize a free resolution of $M$. Since localization is exact, it becomes a free $Q$-resolution of $S^{-1}M$. For a $Q$-module $X$, tensoring this resolution over $Q$ is the same as tensoring the original resolution with $X$ over $A$. The usual localization identity follows by the same argument.
\end{proof}

\begin{lemma}\label{lem:residue2}
Let $(A,\m)$ be local, let $\fd_A M\leq3$, and put $k=A/\m$. If $\m^2$ is flat and $\Tor_3^A(k,M)=0$, then $\Tor_2^A(k,M)=0$.
\end{lemma}
\begin{proof}
Flatness of $\m^2$ implies $\fd_A(A/\m^2)\leq1$. The sequence
\[
0\longrightarrow\m/\m^2\longrightarrow A/\m^2\longrightarrow k\longrightarrow0
\]
therefore gives $\Tor_2^A(\m/\m^2,M)=0$. If $\m/\m^2\ne0$, this module is a nonzero direct sum of copies of $k$, and Tor commutes with direct sums. If $\m/\m^2=0$, then $\m=\m^2$ is flat and $\fd_A k\leq1$.
\end{proof}

\subsection{Filtered colimits and cyclic detection}

\begin{lemma}\label{lem:colimit-details}
For a directed system $(X_\alpha)$ of $A$-modules, an arbitrary $A$-module $M$, and $i\geq0$, the natural map
\[
 \colim_\alpha\Tor_i^A(X_\alpha,M)
 \longrightarrow\Tor_i^A(\colim_\alpha X_\alpha,M)
\]
is an isomorphism. If every element of every module in a directed system $(Y_\alpha)$ maps to zero at some later stage, then $\colim Y_\alpha=0$. It is sufficient that for every stage there be a later stage whose transition map is zero.
\end{lemma}
\begin{proof}
Choose a free resolution $F_\bullet$ of $M$. Tensor products commute with colimits, so
\[
 (\colim X_\alpha)\ot_A F_\bullet
 \cong\colim(X_\alpha\ot_A F_\bullet).
\]
Filtered colimits of modules are exact. They therefore commute with kernels, images, and the quotient defining homology. This gives the isomorphism. An element of a filtered colimit is represented by an element at a single stage; its class is zero exactly when it becomes zero at some later stage. This proves the last assertions.
\end{proof}

\begin{lemma}\label{lem:cyclic-torsion}
Let $S$ be multiplicatively closed, and suppose $\Tor_n^A(U,M)\ne0$ for an $S$-torsion module $U$. Then there exists an ideal $J$ meeting $S$ such that $\Tor_n^A(A/J,M)\ne0$.
\end{lemma}
\begin{proof}
The module $U$ is the directed union of its finitely generated submodules. By Lemma~\ref{lem:colimit-details}, at least one such submodule $U_0$ has nonzero $n$th Tor. Choose generators $u_1,\ldots,u_r$ and set $U_j=Au_1+\cdots+Au_j$. Each quotient $U_j/U_{j-1}$ is cyclic and $S$-torsion. If every quotient had zero $n$th Tor, induction using
\[
 \Tor_n^A(U_{j-1},M)\longrightarrow\Tor_n^A(U_j,M)
 \longrightarrow\Tor_n^A(U_j/U_{j-1},M)
\]
would give zero $n$th Tor for $U_0$. Thus some cyclic quotient works. The annihilator of its generator meets $S$ because that generator is $S$-torsion.
\end{proof}

The use of a cyclic filtration here does not require a composition series: the cyclic quotients can have arbitrary annihilators and need not be simple. Likewise, the transition maps in the final direct system need not be injective. What matters there is their eventual action on Tor.

\subsection{A change-of-rings calculation}

We spell out the spectral sequence used in Proposition~\ref{prop:H}. This also records its convergence and the precise range in which a tensor product occurs as a subquotient of a Tor group.

\begin{proposition}\label{prop:spectral-details}
Let $A\to V$ be a ring homomorphism with $\wgd(V)\leq1$, let $M$ be an $A$-module, and let $U$ be a $V$-module. Put $H_q=\Tor_q^A(V,M)$. For every $n\geq1$ there is a natural short exact sequence
\[
 0\longrightarrow U\ot_V H_n
 \longrightarrow\Tor_n^A(U,M)
 \longrightarrow\Tor_1^V(U,H_{n-1})\longrightarrow0.
\]
In particular, $\Tor_n^A(U,M)=0$ implies $U\ot_V H_n=0$.
\end{proposition}
\begin{proof}
Choose a free $A$-resolution $F_\bullet\to M$ and a flat $V$-resolution
\[
 0\longrightarrow P_1\longrightarrow P_0\longrightarrow U\longrightarrow0.
\]
Such a resolution exists by $\wgd(V)\leq1$. Form the first-quadrant double complex
\[
 C_{p,q}=P_p\ot_V(V\ot_A F_q),\qquad p=0,1,\quad q\geq0.
\]
For each $q$, the module $V\ot_A F_q$ is free over $V$. Taking homology first in the $p$ direction leaves only $U\ot_A F_q$ in degree zero. Thus the homology of the total complex is $\Tor_*^A(U,M)$.

Taking homology first in the $q$ direction is also straightforward: flatness of $P_p$ gives $P_p\ot_V H_q$. Its homology in the $p$ direction is
\[
 E^2_{p,q}=\Tor_p^V(U,H_q).
\]
Only two columns occur. For $r\geq2$, a differential $d_r$ changes the first index by $r$ and hence has zero source or target. The spectral sequence therefore degenerates at the second page. Each total degree has a finite filtration of length at most two, so convergence has no inverse-limit qualification. The filtration in total degree $n$ has successive pieces $E^2_{0,n}$ and $E^2_{1,n-1}$, giving the stated exact sequence.
\end{proof}

\begin{remark}
This proposition concerns the homology of a tensor product complex. It does not assert that $V\ot_A F_\bullet$ is a resolution of $V\ot_A M$. Its higher homology groups are precisely the $H_q$ retained in the formula. In our application the two-column bound comes from the valuation domain $V=A/N$, not from the possibly nonreduced chain ring $A/D$.
\end{remark}

\subsection{Gaussian rings}

We state the results from \cite{BG,DT} in the forms needed below. The numbering in the next paragraph refers to the accessible preprint version \cite{DTpre}; this avoids ambiguity between versions.

\begin{proposition}[Gaussian structure and modulewise Tor inputs]\label{prop:inputs}
Let $(A,\m)$ be a local Gaussian ring, $N=\Nil(A)$, and $D=\{x\in A:x^2=0\}$.
\begin{enumerate}
\item $N$ is the unique minimal prime, $A/N$ is a valuation domain, $D^2=0$, and $A/D$ is a chain ring. If $a\notin D$, then $(0:a)\subseteq D$.
\item For $a,b\in A$, $(a,b)^2$ equals either $a^2A$ or $b^2A$. If $(a,b)^2=a^2A$ and $ab=0$, then $b^2=0$.
\item If $\fd_A M=n\geq1$, then $\Tor_n^A(A/D,M)=0$. If $\Tor_n^A(A/J,M)\ne0$, there exists $J'\supseteq J+D$ with $\Tor_n^A(A/J',M)\ne0$.
\item If every element of $\m$ is a zero divisor and $\fd_A M=n\geq1$, then $\Tor_n^A(A/\m,M)=0$.
\item Suppose every element of $\m$ is a zero divisor, $\m\ne D$, and there are no $0\ne d\in D$ and $a\notin D$ with $(0:d)=aA+D$. Then $\m=\m^2+D$ and $\m^2$ is flat.
\end{enumerate}
\end{proposition}

Parts (i)--(ii) are the standard local structure results recalled in \cite[Theorems 2.5--2.7]{DTpre}. Parts (iii)--(v) are \cite[Lemmas 3.1, 3.3, 3.6, and 3.9]{DTpre}. In particular, part (v) is a structural statement: it has no finite-weak-global-dimension hypothesis. We do not use Lemma 3.10 of that paper to infer flatness of $\m$ for a single finite-dimensional module.

\begin{lemma}[Zero-dimensional case]\label{lem:zero}
If $(B,\mathfrak n)$ is a zero-dimensional local Gaussian ring, then $\Fwd(B)\leq2$.
\end{lemma}
\begin{proof}
Here $\mathfrak n=\Nil(B)$. If $\mathfrak n$ is nilpotent, $B$ is a perfect ring by the standard perfect-ring criterion \cite{Bass}. Thus flat modules are projective, and projective covers give minimal projective resolutions. If a nonzero module had finite positive projective dimension $n$, the last nonzero term in its minimal resolution, tensored with $B/\mathfrak n$, would be nonzero by nilpotent Nakayama. Consequently $\Tor_n^B(B/\mathfrak n,M)\ne0$, contrary to Proposition~\ref{prop:inputs}(iv). Thus in this case the finitistic weak dimension is zero.

Suppose $\mathfrak n$ is not nilpotent. We give the modulewise deduction underlying the remark after \cite[Theorem 4.5]{DTpre}. Assume $\fd_B M=3$, and write $D_B=\{x:x^2=0\}$, $B'=B/D_B$, and $\mathfrak n'=\mathfrak n/D_B$. The modulewise statements \cite[Lemmas 4.2(iii) and 4.4]{DTpre} give, for every nonzero $a\in\mathfrak n'$,
\begin{equation}\label{eq:zero-input}
 \Tor_3^B(B'/aB',M)\ne0,
 \qquad \Tor_3^B(B'/a\mathfrak n',M)=0.
\end{equation}
If $(0:d)=aB+D_B$ for some $0\ne d\in D_B$ and $a\notin D_B$, then $B/(aB+D_B)\cong dB\hookrightarrow B$, contradicting the first assertion of \eqref{eq:zero-input} and Lemma~\ref{lem:top}. Otherwise Proposition~\ref{prop:inputs}(v) implies that $\mathfrak n^2$ is flat. Part (iv) and Lemma~\ref{lem:residue2} give $\Tor_2^B(B/\mathfrak n,M)=0$.

For nonzero $a\in\mathfrak n'$, the cyclic module $aB'/a\mathfrak n'$ is the residue field: its generator cannot belong to $a\mathfrak n'$, since $a=ar$ with $r\in\mathfrak n'$ would imply $a=0$. The exact sequence
\[
0\to aB'/a\mathfrak n'\to B'/a\mathfrak n'\to B'/aB'\to0
\]
now contradicts \eqref{eq:zero-input}. This excludes flat dimension three. Dimension shifting excludes every larger finite value.
\end{proof}

\section{The localization kernel and the obstruction module}\label{sec:kernel}

Until Section~\ref{sec:global}, $A$ is local Gaussian. We use the notation
\[
 S=A\setminus N,\quad Q=A_N,\quad V=A/N,\quad K=\Frac(V),
 \quad I=\ker(A\to Q).
\]
The ring $Q$ is zero-dimensional and local, and $Q/N_Q=K$. In general $A\to Q$ need not be injective.

\begin{lemma}\label{lem:kernel}
The following assertions hold.
\begin{enumerate}
\item For $s\in S$, $N=sN+(0:s)$ and $(0:s)N=0$.
\item $I\subseteq D$, $NI=I^2=0$, and $I$ is a torsion $V$-module.
\item There is an exact sequence $0\to I\to N\to N_Q\to0$.
\item For $s\in S$, $A/(sA+D)\cong V/\bar sV$ and $sI=s^2I$.
\end{enumerate}
Here $\bar s$ denotes the image in $V$.
\end{lemma}
\begin{proof}
Take $x\in N$ and $s\in S$. The equality $(s,x)^2=x^2A$ is impossible because it would make $s^2$ nilpotent. Thus $sx=s^2r$ for some $r\in A$. Reduction modulo $N$ shows $r\in N$. Therefore
\begin{equation}\label{eq:division}
 x=sr+d,\qquad r\in N,\quad d\in(0:s)\subseteq D.
\end{equation}
For $b\in(0:s)$, one has $bx=bsr+bd=0$, since $D^2=0$. This proves (i). Now
\[
 I=\bigcup_{s\in S}(0:s),
\]
so (ii) follows. Each $x/s\in N_Q$ equals $r/1$ by \eqref{eq:division}; this proves surjectivity in (iii), and its kernel is $I$.

Equation \eqref{eq:division} gives $N\subseteq sA+D$, proving the quotient identity. For $i\in I$, the same Gaussian square calculation gives $si=s^2e$ for some $e\in A$. Choose $t\in S$ with $ti=0$. Then $ts^2e=0$, so $e\in I$. This proves $sI\subseteq s^2I$; the reverse inclusion is immediate.
\end{proof}

\subsection{The reduced quotient and a pullback description}

The following consequences of Lemma~\ref{lem:kernel} explain the structure behind the obstruction module. They are not additional hypotheses in the proof.

\begin{proposition}\label{prop:pullback}
Put $B=A/I$. Then $B$ is Gaussian, its nilradical is $N/I$, and every element outside $N/I$ is regular. The inclusion $B\hookrightarrow Q$ identifies $B$ with the inverse image of $V$ under $Q\to K$. Equivalently,
\[
 B\cong Q\times_K V
   =\{(q,v)\in Q\oplus V:q+N_Q=v\text{ in }K\}.
\]
There is an exact sequence of $A$-modules
\[
 0\longrightarrow B\longrightarrow Q\oplus V
 \longrightarrow K\longrightarrow0,
\]
where the last map is $(q,v)\mapsto(q+N_Q)-v$.
\end{proposition}
\begin{proof}
A quotient of a Gaussian ring is Gaussian. Since $I\subseteq N$ and $I$ is nil, the nilradical of $B$ is $N/I$. If $a\notin N$ and $ax\in I$, there exists $s\notin N$ with $sax=0$. Since $sa\notin N$, this means $x\in I$. Thus $a+I$ is regular in $B$.

The quotient map $Q\to Q/N_Q=K$ sends the image of $A$ onto $V\subseteq K$. Conversely, if $q\in Q$ has image $v\in V$, choose $a\in A$ lifting $v$. Then $q-a\in N_Q$. Surjectivity of $N\to N_Q$ gives $n\in N$ with $q-a=n/1$, so $q$ is the image of $a+n$. Hence the inverse image of $V$ is exactly $B$. The fiber-product description follows. Finally $Q\to K$ is surjective; the difference map on $Q\oplus V$ is therefore surjective, with kernel the indicated fiber product.
\end{proof}

This description does not imply that $A$ is a split extension of $B$ by $I$, nor that $A\to B$ preserves flat resolutions. The proof uses only the actual embedding $B\hookrightarrow Q$ to annihilate highest Tor, followed by the short exact sequence $0\to I\to A\to B\to0$. Thus no choice of a coefficient subring or splitting is needed.

\begin{corollary}\label{cor:torsion-square}
For every $s\notin N$, one has $N^2=s^2N^2=sN^2$. Every $S$-torsion $A$-module $U$ is annihilated by $N^2$.
\end{corollary}
\begin{proof}
By Lemma~\ref{lem:kernel}, $N=sN+(0:s)$ and $(0:s)N=0$. Squaring the first equality gives $N^2=s^2N^2$. Since $s^2N^2\subseteq sN^2\subseteq N^2$, all three ideals are equal. If $u\in U$ is killed by $s\in S$ and $n\in N^2$, write $n=sn'$ with $n'\in N^2$. Then $nu=n'su=0$.
\end{proof}

Although the torsion modules in Lemma~\ref{lem:detection} are consequently modules over $A/N^2$, their Tor groups in the argument are computed over $A$. Merely knowing the acting quotient ring does not justify changing the base ring in those Tor groups.

Assume now that $\fd_A M=3$, and write $T_j(X)=\Tor_j^A(X,M)$ and
\[
 H=T_2(V).
\]

\begin{proposition}\label{prop:H}
One has
\begin{equation}\label{eq:H}
 T_3(V)=0,\qquad H\cong T_1(I),\qquad S^{-1}H=0,
 \qquad I\ot_V H=0.
\end{equation}
\end{proposition}
\begin{proof}
By Lemma~\ref{lem:zero}, $\fd_Q M_Q\leq2$. Furthermore
\begin{equation}\label{eq:residueQ}
 \Tor_2^Q(K,M_Q)=0.
\end{equation}
Indeed, this is immediate when $\fd_Q M_Q\leq1$, and otherwise follows from Proposition~\ref{prop:inputs}(iv), since the maximal ideal of $Q$ is nil. The inclusion $V\hookrightarrow K$, Lemmas~\ref{lem:top} and \ref{lem:localtor}, and $\fd_Q M_Q\leq2$ imply $T_3(V)=0$.

The sequences $0\to N\to A\to V\to0$ and $0\to N_Q\to Q\to K\to0$ give
\[
 H\cong T_1(N),\qquad
 T_1(N_Q)=T_2(N_Q)=0.
\]
Apply Tor to Lemma~\ref{lem:kernel}(iii) to obtain $T_1(I)\cong T_1(N)$. Localization of $H$ is the group in \eqref{eq:residueQ}, proving its torsion assertion.

Set $B=A/I$. Since $B\hookrightarrow Q$ and $Q$ is flat over $A$, Lemma~\ref{lem:top} gives $T_3(B)=0$, whence $T_2(I)=0$. Since $NI=0$, the change-of-rings spectral sequence is
\begin{equation}\label{eq:spectral}
 E^2_{p,q}=\Tor_p^V(I,T_q(V))\ \Longrightarrow\ T_{p+q}(I).
\end{equation}
The valuation domain $V$ has weak global dimension at most one. Thus only columns $p=0,1$ occur, all higher differentials are zero, and $E^2_{0,2}=I\ot_V H$ is a subquotient of $T_2(I)=0$.
\end{proof}

\begin{lemma}[Detection away from the nilradical]\label{lem:detection}
There exists $a\in\m\setminus N$ with $T_3(V/\bar aV)\ne0$. In particular $H\ne0$.
\end{lemma}
\begin{proof}
Choose $X$ with $T_3(X)\ne0$. Set $U=\ker(X\to X_Q)$, $Y=\im(X\to X_Q)$, and $C=\operatorname{coker}(X\to X_Q)$. Since $T_3(X_Q)=0$ and $T_4(C)=0$, the sequence $0\to Y\to X_Q\to C\to0$ gives $T_3(Y)=0$. Therefore $0\to U\to X\to Y\to0$ gives $T_3(U)\ne0$.

The module $U$ is $S$-torsion. Tor commutes with filtered colimits, so a finitely generated submodule $U_0$ has nonzero third Tor. A finite filtration obtained by adjoining its generators has cyclic $S$-torsion quotients. The long exact sequences show that at least one such quotient $A/J$ has $T_3(A/J)\ne0$. Its annihilator $J$ meets $S$.

By Proposition~\ref{prop:inputs}(iii), replace $J$ by a larger ideal containing $D$, retaining both properties. Fix $s\in J\cap S$. The ideals between $sA+D$ and $J$ that are finitely generated modulo $D$ form a cofinal directed family in $J$. Since $A/D$ is a chain ring, each is $aA+D$ for some $a\in J$. Such an $a$ cannot belong to $N$, since $s\in aA+D$. Hence
\[
 A/J=\colim_a A/(aA+D)
\]
and some stage has nonzero third Tor. Lemma~\ref{lem:kernel}(iv) identifies that stage with $V/\bar aV$.

Finally, $0\to V\xrightarrow{\bar a}V\to V/\bar aV\to0$ and $T_3(V)=0$ give
\begin{equation}\label{eq:kernelH}
 T_3(V/\bar aV)\cong(0:_H\bar a).
\end{equation}
Thus $H\ne0$.
\end{proof}

\section{Normalization and propagation of highest Tor}\label{sec:normal}

We omit bars on elements acting on $V$-modules when no ambiguity arises.

\begin{lemma}[First localization]\label{lem:first}
After localizing $A$ and $M$ at a prime containing $N$, one may assume that $\fd_A M=3$ and that there is $0\ne\xi\in H$ such that every element of $\m/N$ has a power annihilating $\xi$. Consequently
\begin{equation}\tag{E}\label{eq:E}
 T_3(V/cV)\ne0\quad\text{for every }0\ne c\in\m/N.
\end{equation}
\end{lemma}
\begin{proof}
Choose $0\ne\xi\in H$ and put $\p=\sqrt{\Ann_V(\xi)}$. The radical of a proper ideal in a valuation domain is prime. Moreover $\p\ne0$ because $H$ is torsion. Localize at the inverse image of $\p$. The element $\xi$ survives, since its annihilator is contained in $\p$. Every element of $\p V_{\p}$ has a power killing it.

Localization preserves \eqref{eq:H} and \eqref{eq:kernelH}. For nonzero $c$ in the new maximal ideal, take the least positive integer $r$ with $c^r\xi=0$. Then $c^{r-1}\xi\ne0$ lies in $(0:_H c)$, which proves \eqref{eq:E}. In particular a third Tor remains nonzero, so the localized flat dimension is still three. The nilradical, $D$, and the kernel of localization at the nilradical localize to the corresponding ideals; this follows respectively from localization of nilpotence, square-zero equations, and the exact sequence defining $I$.
\end{proof}

The next elementary observation is the link between the two quotients $A/D$ and $A/N$.

\begin{lemma}[Propagation in a chain ring]\label{lem:chain}
Let $B$ be a chain ring and let $a=cr\ne0$ in $B$. Multiplication by $r$ defines a monomorphism $B/cB\to B/aB$.
\end{lemma}
\begin{proof}
Since $c\notin(0:r)$ and the ideals are comparable, $(0:r)\subseteq cB$. If $rx\in aB=crB$, then $r(x-cy)=0$ for some $y$, and hence $x\in cB$. This is exactly injectivity of the indicated map.
\end{proof}

\begin{proposition}\label{prop:normal}
Assume \eqref{eq:E}. Then
\begin{enumerate}
\item $T_3(A/(aA+D))\ne0$ for every $a\in\m\setminus D$;
\item every element of $\m$ is a zero divisor, and no pair in Proposition~\ref{prop:inputs}(v) exists;
\item $\m^2$ is flat, $\m=\m^2+D=\m^2+I$, and, for $k=A/\m$, $T_2(k)=T_3(k)=0$;
\item $\p=\m/N$ satisfies $\p=\p^2\ne0$ and is not principal.
\end{enumerate}
\end{proposition}
\begin{proof}
Part (i) for $a\notin N$ follows from Lemma~\ref{lem:kernel}(iv). If $a\in N\setminus D$, choose $c\in\m\setminus N$. In $B=A/D$, Lemma~\ref{lem:kernel}(iv) implies $\bar a=\bar c r\ne0$ for some $r\in B$. Lemma~\ref{lem:chain} and Lemma~\ref{lem:top} transfer the nonzero third Tor for $B/\bar cB$ to $B/\bar aB$.

If $(0:d)=aA+D$ with $0\ne d\in D$ and $a\notin D$, then $a\in\m$ and $A/(aA+D)\cong dA\hookrightarrow A$. This contradicts (i). If $c\in\m$ were regular, then $c\notin N$, and \eqref{eq:division} would give $N\subseteq cA$. Hence $V/cV\cong A/cA$ would have flat dimension one, contradicting \eqref{eq:E}. This proves (ii).

Proposition~\ref{prop:inputs}(v) now gives $\m^2$ flat and $\m=\m^2+D$. For a fixed $c\in\m\setminus N$, Lemma~\ref{lem:kernel}(i) gives
\[
 N=cN+(0:c)\subseteq\m^2+I.
\]
Since $I\subseteq D\subseteq N$, this proves $\m=\m^2+I$. Part (iv) of Proposition~\ref{prop:inputs} gives $T_3(k)=0$, and Lemma~\ref{lem:residue2} gives $T_2(k)=0$.

Reducing $\m=\m^2+D$ modulo $N$ gives $\p=\p^2$. The first localization has $\p\ne0$. A nonzero proper principal ideal in a domain cannot be idempotent: if $uV=u^2V$, cancellation makes $u$ a unit. Thus $\p$ is not principal.
\end{proof}

\section{Factoring the obstruction and a second localization}\label{sec:factor}

Retain the element $\xi$ of Lemma~\ref{lem:first} and the hypotheses of Proposition~\ref{prop:normal}.

\begin{lemma}[A nonzero divisible expression]\label{lem:factor}
There exist $b\in\m\setminus N$ and $\eta\in H$ with $\xi=b\eta\ne0$.
\end{lemma}
\begin{proof}
If $I/\m I\ne0$, the surjection $I\to I/\m I$ and $I\ot_V H=0$ imply
\[
 (I/\m I)\ot_V H=0.
\]
The nonzero module $I/\m I$ is a vector space over $k=V/\p$. It follows that $k\ot_V H=0$, or $H=\p H$. Express $\xi$ as a finite sum $\sum b_j\eta_j$ with $b_j\in\p$. Among the nonzero principal ideals $b_jV$, choose the largest, say $bV$. Write each $b_j=br_j$ and combine the sum. A lift of $b$ lies in $\m\setminus N$.

Suppose now that $I=\m I$. Proposition~\ref{prop:normal} gives $\m=\m^2$, so $\m$ is flat and $\fd_A k\leq1$. If $0\ne\lambda\in I$, put $J_\lambda=(0:\lambda)$. The inclusion $A/J_\lambda\cong\lambda A\hookrightarrow A$ and Lemma~\ref{lem:top}, applied to $k$, give
\[
 0=\Tor_1^A(A/J_\lambda,k)=J_\lambda/\m J_\lambda.
\]
Thus
\begin{equation}\label{eq:annfactor}
 J_\lambda=\m J_\lambda.
\end{equation}

Choose $0\to L\to F\to M\to0$ with $F$ free. The isomorphism $H\cong\Tor_1^A(I,M)$ identifies $H$ with
\[
 \ker(I\ot_A L\longrightarrow I\ot_A F).
\]
Write $\xi$ as a finite sum of elementary tensors, express their first factors using $I=\m I$, and discard coefficients in $N$, which act trivially on $I$. Comparability of the remaining finitely many principal ideals modulo $N$ gives
\[
 \xi=az,\qquad a\in\m\setminus N,\quad z\in I\ot_A L.
\]
If $z$ already belongs to the kernel, this is the required expression.

Otherwise let $\lambda_1,\ldots,\lambda_t\in I\setminus\{0\}$ be the nonzero coordinates of its image in $I\ot_A F$. They are finite in number, and $a\lambda_j=0$. Since $N\subseteq J_{\lambda_j}$, equation \eqref{eq:annfactor}, reduced modulo $N$, expresses $\bar a$ as a finite sum of products of elements of $\p$ with elements of $J_{\lambda_j}/N$. Choose a largest principal ideal among the latter factors. Then
\[
 \bar a=\bar b_j\bar c_j,
 \qquad b_j\in\m,\quad c_j\in J_{\lambda_j}.
\]
Both factors have nonzero image in $V$. Choose a smallest principal ideal among $\bar c_1V,\ldots,\bar c_tV$, and let $c$ be the corresponding $c_j$. As each $J_{\lambda_j}$ contains $N$, the element $c$ belongs to every $J_{\lambda_j}$. For the same chosen index, $a-bc\in N$ with $b\in\m\setminus N$.

Consequently $cz$ maps to zero in $I\ot_A F$, while $\xi=b(cz)$ because $N$ annihilates $I\ot_A L$. Taking $\eta=cz$ finishes the proof.
\end{proof}

\subsection{Why finite coefficient comparisons suffice}

The factorization in Lemma~\ref{lem:factor} uses no basis of $I$ or $H$. Every tensor is a finite sum of elementary tensors even when both tensor factors are infinitely generated. Similarly, an element of $I\ot_A F$, with $F$ a direct sum of copies of $A$, has only finitely many nonzero coordinates. This finite support is what makes it possible to choose largest and smallest principal ideals at the two distinct stages of the proof.

For completeness, the elementary coefficient principle used there is recorded separately.

\begin{lemma}\label{lem:coefficient}
Let $(V,\p)$ be a valuation domain, let $J_1,\ldots,J_t$ be ideals satisfying $J_j=\p J_j$, and let $0\ne a\in\bigcap_jJ_j$. There exist $c\in\bigcap_jJ_j$ and $b\in\p$ with $a=bc$.
\end{lemma}
\begin{proof}
For each $j$, write $a=\sum_{\ell=1}^{r_j}b_{j\ell}c_{j\ell}$ with $b_{j\ell}\in\p$ and $c_{j\ell}\in J_j$. Choose a largest principal ideal among the $c_{j\ell}V$, generated by $c_j$. Writing $c_{j\ell}=c_jr_{j\ell}$ gives
\[
 a=c_j\Bigl(\sum_\ell b_{j\ell}r_{j\ell}\Bigr)=c_jb_j,
 \qquad b_j\in\p.
\]
Since $a\ne0$, neither $c_j$ nor $b_j$ is zero. Choose a smallest principal ideal among the finitely many $c_jV$ and call its generator $c$. It belongs to every $J_j$, and the factorization associated with that index gives $a=bc$ with $b\in\p$.
\end{proof}

In Lemma~\ref{lem:factor}, this principle is applied to the ideals $J_{\lambda_j}/N$. The congruence $a-bc\in N$ becomes an exact equality on $I\ot_A L$ because $NI=0$. It is not an equality in $A$ itself, and no cancellation by a zero divisor is being performed.

\begin{lemma}[Torsion in the localization kernel]\label{lem:root}
For $b,\eta$ as above there exist $e\in I$ and an integer $r\geq1$ such that
\begin{equation}\label{eq:root}
 0\ne d=b^2e,\qquad b^rd=0.
\end{equation}
\end{lemma}
\begin{proof}
Since $\xi=b\eta$ is killed by a power of $b$, so is $\eta$, but $b\eta\ne0$. Suppose $bI$ had no nonzero $b$-power-torsion element. Multiplication by $b$ on $bI$ is then injective, and Lemma~\ref{lem:kernel}(iv) makes it surjective. It follows that
\[
 I=(0:_I b)\oplus bI.
\]
Indeed, for $x\in I$, choose $y\in bI$ with $by=bx$; then $x-y\in(0:_I b)$, and the intersection is zero. Applying $\Tor_1^A(-,M)$, the first summand is annihilated by $b$ and multiplication by $b$ is invertible on the second. Thus every $b$-power-torsion element of $H\cong T_1(I)$ is annihilated by $b$, contradicting $b\eta\ne0$.

Choose a nonzero $b$-power-torsion element $d\in bI$. The equality $bI=b^2I$ gives the asserted expression.
\end{proof}

\begin{lemma}[Second localization with preservation of nonvanishing]\label{lem:second}
After another localization and relabeling, \eqref{eq:E} and all the conclusions of Proposition~\ref{prop:normal} still hold, and there are $b\in\m\setminus N$, $e\in I$, and $0\ne d=b^2e$ such that every $x\in\m$ has a power annihilating $d$.
\end{lemma}
\begin{proof}
Before localization put $\q=\sqrt{\Ann_V(d)}$. This is a proper prime containing the nonzero element $b$. Localize at its inverse image in $A$. The element $d$ survives. Elements of $N$ already annihilate $d$, and every element of the new maximal ideal modulo $N$ has a power killing it.

We verify \eqref{eq:E} instead of assuming that a previously nonzero Tor element survives. It suffices to take $0\ne c\in\q\subseteq V$, since denominators outside $\q$ are units. If $c\in bV$, use $\eta$. It is killed by a power of $c$. It survives at $\q$: if $s\notin\q$ and $s\eta=0$, then comparability in $V$ gives $b\in sV$, forcing $b\eta=0$.

In the other case write $b=cu$ and use $\eta_c=u\eta$. One has $c\eta_c=\xi\ne0$. The original normalization says that a power of $c$ kills $\xi$, since $c$ belongs to the original maximal ideal. Thus $\eta_c$ is $c$-power torsion. If $s\notin\q$ killed $\eta_c$, then $c\in sV$ would force $c\eta_c=0$, a contradiction. Hence $\eta_c$ also survives.

In both cases the localized $H$ contains a nonzero $c$-power-torsion element, and therefore a nonzero element killed by $c$. The localized form of \eqref{eq:kernelH} gives \eqref{eq:E}. The flat dimension consequently remains three. Proposition~\ref{prop:normal} applies afresh to the localized ring and module. Equation \eqref{eq:root} persists, with $d\ne0$. Notice that this argument does not assert that $\xi$ itself survives the second localization.
\end{proof}

\section{The annihilator factorization and the local contradiction}\label{sec:limit}

We now work with the data after Lemma~\ref{lem:second}. Put $\p=\m/N$ and $k=A/\m$. Thus $\p=\p^2\ne0$, $\p$ is not principal, $T_2(k)=0$, and $0\ne d=b^2e$ with $e\in I$ is power-torsion under every element of $\m$.

\begin{lemma}\label{lem:transition}
Suppose
\[
 bV\subseteq u_1V\subsetneq u_2V\subsetneq\p.
\]
Then the natural map
\[
 T_3(A/(bu_1A+N))\longrightarrow T_3(A/(bu_2A+N))
\]
is zero. Representatives $b,u_1,u_2$ are chosen in $A$.
\end{lemma}
\begin{proof}
The Gaussian square property for $u_1,u_2$ must give $(u_1,u_2)^2=u_2^2A$. Otherwise reduction modulo $N$ would imply $u_2^2\in u_1^2V$, contradicting $u_1V\subsetneq u_2V$. Thus $u_1u_2=xu_2^2$ for some $x\in A$. Write
\[
 u_1=xu_2+\delta,\qquad u_2\delta=0.
\]
Reduction modulo $N$ shows $x\in\m\setminus N$. Also $\delta\in I$. Since $b\in u_2A+N$ and $N\delta=0$, one has $b\delta=0$, whence
\begin{equation}\label{eq:bu}
 bu_1=xbu_2.
\end{equation}
It follows that $(0:bu_2)\subseteq(0:bu_1)$.

This inclusion is strict. If the annihilators were equal, multiplication by $x$ would be injective on $bu_2A$: for $xbu_2r=0$, equation \eqref{eq:bu} implies $r\in(0:bu_1)=(0:bu_2)$. On the other hand $d\in bu_2A$. Indeed, write $b=u_2v+n$ with $n\in N$; since $Ne=0$,
\[
 d=b^2e=bu_2ve.
\]
Some power of $x\in\m$ kills the nonzero element $d$, contradicting injectivity on the ideal $bu_2A$.

Choose $\lambda\in(0:bu_1)\setminus(0:bu_2)$. As $bu_1\notin N$, this element belongs to $I$, and $N\lambda=0$. Consequently $(0:\lambda)$ contains $N$, and its image is an ideal of $V$. Comparability of ideals of $V$ and $bu_2\notin(0:\lambda)$ give
\begin{equation}\label{eq:sandwich}
 bu_1A+N\subseteq(0:\lambda)\subseteq bu_2A+N.
\end{equation}
The projection between the cyclic modules therefore factors through
\[
 A/(0:\lambda)\cong\lambda A\hookrightarrow A.
\]
Lemma~\ref{lem:top} makes the third Tor of this intermediate module zero, proving the claim.
\end{proof}

\begin{proposition}[Local exclusion]\label{prop:local}
No local Gaussian ring admits a module of flat dimension three.
\end{proposition}
\begin{proof}
Suppose such a module exists. The preceding sections supply the normalized data. The principal ideals $uV$ satisfying
\[
 bV\subseteq uV\subsetneq\p
\]
form a directed family with union $\p$. Indeed, comparability gives directedness and cofinality, and no element generates $\p$. Each member is properly contained in another member. Passing to inverse images in $A$ gives
\begin{equation}\label{eq:limit}
 A/(b\m+N)=\colim_u A/(buA+N).
\end{equation}
For every stage, Lemma~\ref{lem:transition} supplies a later stage at which the induced third Tor map is zero. Since Tor commutes with filtered colimits,
\[
 T_3(A/(b\m+N))=0.
\]
In the domain $V$, multiplication by $b$ identifies $V/\p$ with $bV/b\p$. Hence there is an exact sequence
\[
 0\longrightarrow k\longrightarrow A/(b\m+N)
 \longrightarrow A/(bA+N)\longrightarrow0.
\]
The vanishing just proved and $T_2(k)=0$ imply $T_3(A/(bA+N))=0$. But $A/(bA+N)\cong V/bV$, contradicting \eqref{eq:E}.
\end{proof}

\subsection{The maps in the final exact sequence}

We describe the last short exact sequence at the level of elements. This makes explicit the final use of cancellation in the domain $V$. The maps are
\[
 \begin{aligned}
 \alpha:V/\p&\longrightarrow V/b\p,&
       v+\p&\longmapsto bv+b\p,\\
 \beta:V/b\p&\longrightarrow V/bV,&
       v+b\p&\longmapsto v+bV.
 \end{aligned}
\]
The first map is well defined because $b\p$ contains $b(v-v')$ whenever $v-v'\in\p$. If $bv\in b\p$, cancellation of the nonzero element $b$ in $V$ gives $v\in\p$, so $\alpha$ is injective. The second map is surjective, and its kernel is $bV/b\p$, exactly the image of $\alpha$. Thus
\[
 0\longrightarrow V/\p\xrightarrow{\alpha}V/b\p
 \xrightarrow{\beta}V/bV\longrightarrow0
\]
is exact. Under the quotient maps from $A$, these are precisely the three modules in the last paragraph of Proposition~\ref{prop:local}.

The relevant portion of the long exact sequence is
\[
 \Tor_3^A(V/b\p,M)\longrightarrow\Tor_3^A(V/bV,M)
 \longrightarrow\Tor_2^A(V/\p,M).
\]
The outside terms vanish for different reasons: the first by the filtered-colimit argument and the second by flatness of $\m^2$ together with highest residue-field Tor vanishing. Neither of these reasons alone suffices. Their combination forces the middle term to vanish, in direct conflict with the normalization condition (E).

\section{The global theorem and consequences}\label{sec:global}

\begin{proof}[Proof of Theorem~\ref{thm:main}, apart from sharpness]
For a local Gaussian ring $A$, a module of finite flat dimension $n>3$ has an $(n-3)$rd syzygy of flat dimension exactly three. Proposition~\ref{prop:local} therefore gives $\Fwd(A)\leq2$.

Let $R$ be Gaussian and $M$ an $R$-module of finite flat dimension. Each $R_{\m}$ is Gaussian and $\fd_{R_{\m}}M_{\m}\leq\fd_R M$, so the local result gives $\fd_{R_{\m}}M_{\m}\leq2$ for every maximal ideal $\m$. Take a free presentation through degree one,
\[
0\longrightarrow L\longrightarrow F_1\longrightarrow F_0\longrightarrow M\longrightarrow0.
\]
Then every $L_{\m}$ is flat. Flatness is local, so $L$ is flat and $\fd_R M\leq2$.
\end{proof}

\begin{corollary}\label{cor:values}
If $R$ is Gaussian and $M$ is any $R$-module, then
\[
 \fd_R M\in\{0,1,2,\infty\}.
\]
For a nonzero Gaussian ring, $\Fwd(R)\in\{0,1,2\}$.
\end{corollary}

\begin{corollary}\label{cor:criterion}
Let $R$ be Gaussian and $M$ an $R$-module. The following are equivalent:
\begin{enumerate}
\item $M$ has finite flat dimension;
\item $\fd_R M\leq2$;
\item $\Tor_3^R(R/J,M)=0$ for every finitely generated ideal $J$ of $R$.
\end{enumerate}
\end{corollary}
\begin{proof}
Only (iii)$\Rightarrow$(ii) needs explanation. For a second syzygy $L$ of $M$, dimension shifting gives $\Tor_1^R(R/J,L)=0$ for every finitely generated ideal $J$. The ideal criterion for flatness implies that $L$ is flat.
\end{proof}

\begin{corollary}\label{cor:regular}
If $A$ is local Gaussian and every element outside $N=\Nil(A)$ is regular, then $\Fwd(A)\leq2$.
\end{corollary}
\begin{proof}
There is also a short proof independent of Sections~\ref{sec:normal}--\ref{sec:limit}. In this case $I=0$. If $\fd_A M=3$, Proposition~\ref{prop:H} gives $H=0$, whereas Lemma~\ref{lem:detection} gives $H\ne0$. Dimension shifting finishes the argument.
\end{proof}

\begin{remark}
Theorem~\ref{thm:main} is a bound on finite flat dimensions, not a finite bound on weak global dimension. In particular it is compatible with the infinite weak global dimension of nonreduced Gaussian rings established in \cite{DT}. For reduced Gaussian rings, the stronger weak-global-dimension bound one is known; see \cite{Glaz05}. No claim that the three possible finitistic values are characterized solely by the nilradical is made here.
\end{remark}

\subsection{Locality of the invariant}

The local-to-global passage also gives the following equality, which is useful independently of the Gaussian hypothesis.

\begin{proposition}\label{prop:locality}
For every commutative ring $R$,
\[
 \Fwd(R)=\sup_{\m\in\Max(R)}\Fwd(R_{\m}).
\]
The same supremum may be taken over all prime ideals.
\end{proposition}
\begin{proof}
First suppose $X$ is an $R_{\p}$-module. A flat $R_{\p}$-module is flat over $R$, because $R_{\p}$ is flat over $R$ and the composite of two exact tensor functors is exact. A flat resolution over $R_{\p}$ is therefore a flat resolution over $R$. Conversely, localization of an $R$-flat resolution of $X$ gives an $R_{\p}$-flat resolution, and $X_{\p}\cong X$. Hence
\[
 \fd_R X=\fd_{R_{\p}}X.
\]
It follows that every local finitistic weak dimension is at most $\Fwd(R)$.

For the converse, if the supremum over maximal ideals is infinite there is nothing to prove. If it is bounded by an integer $d$, let $M$ have finite flat dimension over $R$. Then $\fd_{R_{\m}}M_{\m}\leq d$ for every maximal ideal. A $d$th syzygy $L$ in a free resolution localizes to a flat module everywhere. To see that $L$ is flat, apply localization to $\Tor_1^R(R/J,L)$ for every finitely generated ideal $J$. All its maximal localizations are zero, so the module itself is zero. The ideal criterion now gives flatness and $\fd_R M\leq d$.

A module whose localizations at all maximal ideals vanish is zero: a nonzero element has a proper annihilator contained in a maximal ideal and survives at that ideal. This justifies the vanishing assertion used above. The equality using all primes follows from the two inequalities and the inclusion of maximal ideals among prime ideals.
\end{proof}

\begin{corollary}
A finite product $R=R_1\times\cdots\times R_r$ of nonzero Gaussian rings satisfies
\[
 \Fwd(R)=\max_{1\leq j\leq r}\Fwd(R_j)\leq2.
\]
\end{corollary}
\begin{proof}
An $R$-module decomposes as the finite direct sum of its components under the central idempotents of the product. Tensor products and exactness are componentwise, so its flat dimension is the maximum of the flat dimensions of its components. A finite flat dimension occurs precisely when all these dimensions are finite. Taking suprema proves the assertion. Gaussianity of the product follows by the same componentwise interpretation of polynomial contents.
\end{proof}

\section{An explicit sharp example}\label{sec:example}

We give an example together with a direct calculation, rather than obtaining sharpness only from the classification of chain rings in \cite{Couchot12}.

\begin{example}\label{ex:sharp}
Let $k$ be a field, let
\[
 F=\bigcup_{n\geq1}k((t^{1/n!})),
\]
with the compatible $t$-adic valuation $v:F^\times\to\mathbb Q$, and set
\[
 W=\{0\}\cup\{f\in F^\times:v(f)\geq0\},\qquad
 J=\{0\}\cup\{f\in W\setminus\{0\}:v(f)>1\}.
\]
Then $A=W/J$ is a noncoherent local Gaussian total ring of quotients with $\Fwd(A)=2$.
\end{example}
\begin{proof}
The union $F$ is a field with value group $\mathbb Q$, and $W$ is its valuation ring. The quotient $A$ is a chain ring and hence Gaussian. A nonzero element of $A$ has a well-defined value in $\mathbb Q\cap[0,1]$: adding an element of $J$ does not change that value. Its maximal ideal is
\[
 \m=\{0\}\cup\{\bar f:0<v(f)\leq1\}.
\]
Every element of positive value is nilpotent, and elements of value zero are units. Thus $A$ is a total ring of quotients.

We first prove that $\m$ is flat. Since every finitely generated ideal of $A$ is principal, it suffices to show that $(r)\ot_A\m\to\m$ is injective for every $r\in A$. The presentation $(r)\cong A/(0:r)$ identifies its kernel with
\[
 \frac{\{x\in\m:rx=0\}}{(0:r)\m}.
\]
The cases $r=0$ and $r$ a unit are immediate. Otherwise write $v(r)=\beta>0$. If $0\ne x\in\m$ and $rx=0$, put $v(x)=\alpha>0$. Then $\alpha+\beta>1$. Choose rational $\varepsilon$ with
\[
 0<\varepsilon<\min\{\alpha,\alpha+\beta-1\}.
\]
For a representative of $x$, divide by $t^\varepsilon$ in $W$ and call its image $c\in A$. Then
\[
 x=c\,\overline{t^\varepsilon},\qquad
 rc=0,\qquad \overline{t^\varepsilon}\in\m.
\]
Hence $x\in(0:r)\m$, proving flatness.

Put $x=\bar t\ne0$ and $\kk=A/\m$. One has $(0:x)=\m$ and $xA\cong\kk$. The module $\kk$ is not flat: tensoring the inclusion $xA\hookrightarrow A$ with $\kk$ gives a zero map from the nonzero module $xA\ot_A\kk\cong\kk$ to $\kk$. Since $\m$ is flat, $\fd_A\kk=1$.

The sequence
\[
 0\longrightarrow xA\longrightarrow A\longrightarrow A/xA\longrightarrow0
\]
now implies $\fd_A(A/xA)=2$. Indeed, its dimension-shifting isomorphisms identify $\Tor_2^A(-,A/xA)$ with $\Tor_1^A(-,xA)$, which is not identically zero. Theorem~\ref{thm:main} gives the matching upper bound.

Finally $\m$ is not finitely generated: any finite set of nonzero elements has a smallest positive value, whereas $\m$ contains elements of smaller positive value. Since $\m=(0:x)$, the annihilator of an element is not finitely generated; therefore $A$ is not coherent.
\end{proof}

Fields realize finitistic weak dimension zero. A discrete valuation domain realizes the value one: its weak global dimension is one, and a quotient by a nonzero nonunit has flat dimension one. Together with Example~\ref{ex:sharp}, this realizes all values in Corollary~\ref{cor:values}.

\begin{remark}

The strict inequality in the definition $J=\{f:v(f)>1\}$ is essential to the calculation. It leaves the element $x=\bar t$, of value exactly one, nonzero. If $0\ne r\in A$ has value $\beta\in(0,1]$, then
\[
 (0:r)=\{0\}\cup\{a\in A\setminus\{0\}:v(a)>1-\beta\}.
\]
Indeed, a product of nonzero representatives maps to zero exactly when the sum of their values exceeds one. For $\beta=1$, this is the maximal ideal $\m$; for $\beta<1$, it is an open final segment of the value set. By contrast, the nonzero principal ideal generated by an element of value $\gamma$ is the closed segment
\[
 rA=\{0\}\cup\{a\ne0:v(a)\geq\gamma\},\qquad v(r)=\gamma.
\]
This distinction accounts for the non-finite generation of the annihilators in the example. No finite set of positive-valued generators can produce an open segment with no smallest value.

The maximal ideal is idempotent. If $0\ne a\in\m$ has value $\alpha$, divide its representative by $t^{\alpha/2}$ to obtain
\[
 a=\overline{t^{\alpha/2}}\,c,
 \qquad v(c)=\alpha/2>0.
\]
Thus $\m\subseteq\m^2$, and the reverse inclusion is automatic. Consequently
\[
 \Tor_1^A(\kk,\kk)\cong\m/\m^2=0,
\]
although $\kk$ is not flat. This provides a concrete warning that a single residue-field Tor test cannot replace the ideal criterion for flatness over a non-Noetherian local ring. In the main proof, the residue-field test is used only in conjunction with the additional structural and cyclic-module arguments.

There is also an explicit flat resolution of the module realizing dimension two:
\[
 0\longrightarrow\m\longrightarrow A
 \xrightarrow{\;x\;}A\longrightarrow A/xA\longrightarrow0.
\]
Exactness at the first copy of $A$ is the equality $(0:x)=\m$, and exactness at the second is the definition of $xA$. The first term is flat by the valuation calculation, but $xA\cong\kk$ is not flat. Therefore the displayed resolution cannot be shortened to length one. This is a direct homological explanation of sharpness.

The same construction works with any positive rational endpoint $\gamma$ in place of one, taking the ideal $\{f:v(f)>\gamma\}$ and the element $\overline{t^\gamma}$. The flatness argument uses a rational $\varepsilon$ smaller than both the value of the chosen element and its excess above the annihilation threshold. Thus the example is a family of valuation-cut constructions, not a consequence of a special property of the particular exponent one.
\end{remark}

\end{document}